\documentclass[a4paper,11pt]{article}

\usepackage[T1]{fontenc}
\usepackage[cp1250]{inputenc}
\usepackage{lmodern}

\usepackage[english]{babel}
\usepackage{latexsym}
\usepackage{indentfirst}
\usepackage{dsfont}
\usepackage{amsmath}
\usepackage{amsthm}
\usepackage{amsfonts}
\usepackage{stackrel}
\usepackage{tikz}
\usepackage{wasysym}
\usepackage{authblk}

\newtheorem{thm}{Theorem}
\newtheorem{lemma}[thm]{Lemma}

\newcommand{\reals}{\mathbb{R}}

\newcommand{\complex}{\mathbb{C}}

\begin{document}
\author{Mateusz Krukowski}
\affil{\L\'od\'z University of Technology, Institute of Mathematics, \\ W\'ol\-cza\'n\-ska 215, \
90-924 \ \L\'od\'z, \ Poland}
\title{Images of circles, lines, balls and half-planes under M\"{o}bius transformations}
\maketitle
\tableofcontents
\date{}

\section{Introduction}

The idea of the paper arose during complex analysis classes led by the author. During classes, students quickly grew tired of constantly looking for the images of circles, lines, balls and half-planes under \textit{M\"{o}bius transformations}, i.e. functions $h:\overline{\complex}\longrightarrow\overline{\complex}$ of the form
$$h(z) = \frac{az+b}{cz+d}, \hspace{0.4cm}\text{where }\ a,b,c,d\in \complex \hspace{0.4cm}\text{and }\ bc - ad \neq 0.$$

\noindent
However, such a problem was virtually certain to appear on the final exam, composed by the main lecturer. The author decided to write this text in order facilitate the comprehension of the subject. In fact, it is the author's strong belief that the paper will be of considerable value for all students enrolled in complex analysis classes. 

It is a common practice to refer to circles and lines together as \textit{generalized circles}. The idea behind this terminology is that stereographic images of lines onto the Riemann sphere $\overline{C}$ are circles. The most popular result relating generalized circles and M\"{o}bius transformations is the following:

\begin{thm}(\cite{Chadzynski}, p. 32 or \cite{Gamelin}, p. 65)\\
Any M\"{o}bius transformation maps a generalized circle onto a generalized circle.
\end{thm}

Despite the stunning beauty of the above theorem, it does not specify under what conditions the circle is mapped onto a circle or a line etc.  In the sequel, we will attempt to correct this nagging flaw.

For $c=0$ the M\"{o}bius transformation reduces to an \textit{affine transformation} (in that scenario $d\neq 0$). If $c\neq 0$, we can write
$$\forall_{z\in\overline{\complex}}\ h(z) = \frac{a}{c} + \frac{bc-ad}{c}\cdot \frac{1}{cz+d},$$

\noindent
which proves that any M\"{o}bius transformation is a composition of affine transformations and the inversion map $z\mapsto z^{-1}$. Calculating the images of sets under affine transformations is a trivial issue, so in the following chapters we focus our attention only on the inversion function.

\section{Inverting circles and balls}

Until the end of the paper $\iota:\overline{\complex}\longrightarrow\overline{\complex}$ will denote the \textit{inversion}, i.e. 
$$\iota(z) = z^{-1}, \hspace{0.4cm}\text{where}\hspace{0.4cm} \iota(0) = \infty,\ \iota(\infty) = 0.$$ 

\noindent
As the title suggests, the current chapter focuses on finding the images of circles and balls under $\iota.$ Theorem \ref{imageofcircle} below describes the image of an arbitrary circle under the inversion.

\begin{thm}
Let $S(z_*,R)$ be a circle, centered at $z_*$ with radius $R>0$.
\begin{itemize}
	\item If $0\not\in S(z_*,R)$, then $\iota(S(z_*,R))$ is a circle described by the equation:
	$$\left(x - \frac{-x_*}{R^2 - (x_*^2 + y_*^2)}\right)^2 + \left(y - \frac{y_*}{R^2 - (x_*^2 + y_*^2)}\right)^2 = \left(\frac{R}{R^2-(x_*^2+y_*^2)}\right)^2.$$
	\item If $0\in S(z_*,R)$, then $\iota(S(z_*,R))$ is a line described by the equation
	$$x_*x-y_*y = \frac{1}{2}.$$
\end{itemize}
\label{imageofcircle}
\end{thm}
\begin{proof}
We parametrize the circle $S(z_*,R)$ by
$$z = z_* + Re^{i\phi}, \hspace{0.4cm}\text{where}\hspace{0.4cm} \phi\in [0,2\pi].$$

\noindent
Let us observe that  
\begin{gather}
\frac{1}{z} = \frac{1}{z_*+Re^{i\phi}} = \frac{\overline{z_*}+Re^{-i\phi}}{|z_*+Re^{i\phi}|^2} = \frac{(x_*+R\cos(\phi)) - i(y_*+R\sin(\phi))}{(x_*+R\cos(\phi))^2 +(y_*+R\sin(\phi))^2},
\label{1przezokreg}
\end{gather}

\noindent
so we define $u,v : [0,2\pi]\longrightarrow\reals$ by
\begin{gather*}
u(\phi) := x_* + R\cos(\phi),\\
v(\phi) := y_* + R\sin(\phi).
\end{gather*}

\noindent
In the sequel, for convenience of the notation, we simply write $u$ and $v$, neglecting the dependence on the angle $\phi$. Equations (\ref{1przezokreg}) lead to
$$\iota(z) = \frac{1}{z} = \frac{u-iv}{u^2 + v^2}.$$

At first, we analyze the scenario, when $0\not\in S(z_*,R)$. This is equivalent to \mbox{$x_*^2+y_*^2 \neq R^2$}. Observe that for every $a,b\in\reals$ we have
\begin{gather}
\left(\frac{u}{u^2+v^2} - a\right)^2 + \left(\frac{-v}{u^2+v^2} - b\right)^2 = \frac{2bv-2au+1}{u^2+v^2} + a^2 + b^2.
\label{rownosciab}
\end{gather}

\noindent
We put
$$a = \frac{-x_*}{R^2 - (x_*^2+y_*^2)} \hspace{0.4cm}\text{and}\hspace{0.4cm} b = \frac{y_*}{R^2 - (x_*^2+y_*^2)},$$

\noindent
so that the numerator in the fraction on the right-hand side of (\ref{rownosciab}) has the following form:
\begin{equation}
\begin{split}
&2\cdot\frac{y_*}{R^2-(x_*^2+y_*^2)} \cdot v - 2\cdot \frac{-x_*}{R^2-(x_*^2+y_*^2)} \cdot u + 1 \\
&= \frac{2y_*(y_*+R\sin(\phi)) + 2x_*(x_*+R\cos(\phi)) + R^2-(x_*^2+y_*^2)}{R^2-(x_*^2+y_*^2)} \\
&= \frac{x_*^2 + y_*^2 + 2R\bigg(x_*\cos(\phi) + y_*\sin(\phi)\bigg) + R^2}{R^2-(x_*^2+y_*^2)} = \frac{u^2+v^2}{R^2-(x_*^2+y_*^2)}.
\end{split}
\label{2bv2au1}
\end{equation}

\noindent
By (\ref{rownosciab}) and (\ref{2bv2au1}) we gather that
\begin{gather*}
\left(\frac{u}{u^2+v^2} - \frac{-x_*}{R^2 - (x_*^2 + y_*^2)}\right)^2 + \left(\frac{-v}{u^2+v^2} - \frac{y_*}{R^2 - (x_*^2 + y_*^2)}\right)^2 \\
= \frac{1}{R^2-(x_*^2+y_*^2)} + \left(\frac{-x_*}{R^2-(x_*^2+y_*^2)}\right)^2 + \left(\frac{y_*}{R^2-(x_*^2+y_*^2)}\right)^2 = \left(\frac{R}{R^2-(x_*^2+y_*^2)}\right)^2,
\end{gather*}

\noindent
which proves the first part of the theorem.

As far as the second part of the theorem is concerned, we assume that $0\in S(z_*,R)$. This is equivalent to \mbox{$x_*^2+y_*^2 = R^2$}. Observe that for every $a,b,c\in\reals$, the following equalities hold:
\begin{equation}
\begin{split}
&a\frac{u}{u^2+v^2} + b\frac{-v}{u^2+v^2} + c \\
&= \frac{a(x_*+R\cos(\phi)) - b(y_*+R\sin(\phi)) + c\bigg(2R^2 + 2R(x_*\cos(\phi) + y_*\sin(\phi))\bigg)}{2R^2 + 2R\bigg(x_*\cos(\phi) + y_*\sin(\phi)\bigg)} \\
&=\frac{\bigg(ax_*-by_*+2cR^2\bigg) + (a + 2cx_*)R\cos(\phi) + (-b +2cy_*)R\sin(\phi)}{2R^2 + 2R\bigg(x_*\cos(\phi) + y_*\sin(\phi)\bigg).}
\end{split}
\label{aubvc}
\end{equation}

\noindent
For $a=x_*,\ b=-y_*,$ and $c = -\frac{1}{2},$ the equalities (\ref{aubvc}) imply the second part of the theorem, since the numerator vanishes. This concludes the proof. 
\end{proof}

In the sequel, $B(z_*,R),\ \overline{B}(z_*,R)$ and $S(z_*,R)$ denote the open ball, the closed ball and the circle centered at $z_*$ with radius $R>0$, respectively. The next technical result will be crucial in further considerations. Intuitively, it says the if $0$ does not belong to the ball $\overline{B}(z_*,R)$, then the interior of the original ball is mapped into the interior of the image ball. If, however, $0$ lies inside the ball $B(z_*,R)$, then the interior of the original ball is mapped into the exterior of the image ball.

\begin{lemma}
Let $z\in B(z_*,R)$. 
\begin{itemize}
	\item If $0\not\in \overline{B}(z_*,R)$, then 
	$$\iota(z) \in B\left(\frac{-x_*}{R^2-(x_*^2+y_*^2)}+i\frac{y_*}{R^2-(x_*^2+y_*^2)},\frac{R}{R^2-(x_*^2+y_*^2)}\right).$$
	\item If $0\in B(z_*,R)$, then 
	$$\iota(z) \not\in \overline{B}\left(\frac{-x_*}{R^2-(x_*^2+y_*^2)}+i\frac{y_*}{R^2-(x_*^2+y_*^2)},\frac{R}{R^2-(x_*^2+y_*^2)}\right).$$
\end{itemize}
\label{ballispreserved}
\end{lemma}
\begin{proof}
Since
$$\iota(z) = \frac{x}{x^2+y^2} + i\frac{-y}{x^2+y^2},$$

\noindent 
then
\begin{equation}
\begin{split}
&\left(\frac{x}{x^2+y^2} - \frac{-x_*}{R^2-(x_*^2+y_*^2)}\right)^2 + \left(\frac{-y}{x^2+y^2} - \frac{y_*}{R^2-(x_*^2+y_*^2)}\right)^2 \\
&= \frac{\bigg(x(R^2-(x_*^2+y_*^2)) + x_*(x^2+y^2)\bigg)^2 + \bigg(y(R^2-(x_*^2+y_*^2)) + y_*(x^2+y^2)\bigg)^2}{(x^2+y^2)^2\bigg(R^2 - (x_*^2+y_*^2)\bigg)^2}\\
&= \frac{ (x^2+y^2)\bigg(R^2-(x_*^2+y_*^2)\bigg)^2 + 2(x^2+y^2)\bigg(R^2-(x_*^2+y_*^2)\bigg)(xx_*+yy_*) + (x_*^2+y_*^2)(x^2+y^2)^2}{(x^2+y^2)^2\bigg(R^2 - (x_*^2+y_*^2)\bigg)^2} \\
&= \frac{ \bigg(R^2-(x_*^2+y_*^2)\bigg)^2 + 2\bigg(R^2-(x_*^2+y_*^2)\bigg)(xx_*+yy_*) + (x_*^2+y_*^2)(x^2+y^2)}{(x^2+y^2)\bigg(R^2 - (x_*^2+y_*^2)\bigg)^2}.
\end{split}
\label{iotazgdziejest}
\end{equation}

\noindent
Let us observe that $z = x+iy\in B(z_*,R)$ is equivalent to 
\begin{gather}
R^2 > (x-x_*)^2 + (y-y_*)^2 \ \Longleftrightarrow \ \bigg(R^2-(x_*^2+y_*^2)\bigg) + 2(xx_*+yy_*) > x^2+y^2.
\label{r2wiekszeniz}
\end{gather}

\begin{itemize}
	\item At first, we consider the situation when $0\not\in \overline{B}(z_*,R)$. Consequently, \mbox{$R^2-(x_*^2+y_*^2) < 0$} which implies that (\ref{r2wiekszeniz}) is equivalent to
\begin{gather*}
\bigg(R^2-(x_*^2+y_*^2)\bigg)^2 + 2\bigg(R^2-(x_*^2+y_*^2)\bigg)(xx_*+yy_*) < \bigg(R^2-(x_*^2+y_*^2)\bigg)(x^2+y^2) \\
\Longleftrightarrow \ \bigg(R^2-(x_*^2+y_*^2)\bigg)^2 + 2\bigg(R^2-(x_*^2+y_*^2)\bigg)(xx_*+yy_*) + (x_*^2+y_*^2)(x^2+y^2) < (x^2+y^2)R^2.
\end{gather*}

\noindent
Using the above estimate to equalities (\ref{iotazgdziejest}), we conclude that 
\begin{gather*}
\left(\frac{x}{x^2+y^2} - \frac{-x_*}{R^2-(x_*^2+y_*^2)}\right)^2 + \left(\frac{-y}{x^2+y^2} - \frac{y_*}{R^2-(x_*^2+y_*^2)}\right)^2 < \left(\frac{R}{R^2-(x_*^2+y_*^2)}\right)^2,
\end{gather*}

\noindent
which proves the first part of the theorem.

	\item If $0\in B(z_*,R)$ then $R^2 - (x_*^2+y_*^2) > 0$ and the reasoning is analogous.
\end{itemize}

\end{proof}

\begin{thm}
\noindent
\begin{itemize}
	\item If $0\not\in \overline{B}(z_*,R)$ then 
	$$\iota(B(z_*,r)) = B\left(\frac{-x_*}{R^2-(x_*^2+y_*^2)}+i\frac{y_*}{R^2-(x_*^2+y_*^2)},\frac{R}{R^2-(x_*^2+y_*^2)}\right).$$ 
	\item If $0 \in B(z_*,R)$ then 
	$$\iota(B(z_*,r)) = \overline{\complex} \backslash \overline{B}\left(\frac{-x_*}{R^2-(x_*^2+y_*^2)}+i\frac{y_*}{R^2-(x_*^2+y_*^2)},\frac{R}{R^2-(x_*^2+y_*^2)}\right).$$ 
\end{itemize}
\end{thm}
\begin{proof}
As in Lemma \ref{ballispreserved}, the argument in both cases bears a striking resemblance. Hence, we focus solely on the first part of the theorem. By Lemma \ref{ballispreserved}, we already know the following inclusion:
$$\iota(B(z_*,r)) \subset B\left(\frac{-x_*}{R^2-(x_*^2+y_*^2)}+i\frac{y_*}{R^2-(x_*^2+y_*^2)},\frac{R}{R^2-(x_*^2+y_*^2)}\right).$$

\noindent
We will prove the reverse inclusion. 

Let the point $u+iv$ belong to an open ball
$$B\left(\frac{-x_*}{R^2-(x_*^2+y_*^2)}+i\frac{y_*}{R^2-(x_*^2+y_*^2)},\frac{R}{R^2-(x_*^2+y_*^2)}\right).$$

\noindent
We put
$$x = \frac{u}{u^2+v^2} \hspace{0.4cm}\text{and}\hspace{0.4cm} y = \frac{-v}{u^2+v^2}.$$

\noindent
Consequently, we have
$$\iota(x+iy) = \frac{x}{x^2+y^2} + i\frac{-y}{x^2+y^2} = u+iv,$$

\noindent
and it remains to establish that $x+iy \in B(z_*,R)$. 

Recall that we are dealing with the situation when $0\not\in B(z_*,R)$, i.e. \mbox{$R^2-(x_*^2+y_*^2)<0$}. Let us observe the following estimates
\begin{equation}
\begin{split}
&\left(\frac{R}{R^2 - (x_*^2+y_*^2)}\right)^2 > \left(u-\frac{-x_*}{R^2-(x_*^2+y_*^2)}\right)^2 + \left(v - \frac{y_*}{R^2-(x_*^2+y_*^2)}\right)^2 \\
&\Longleftrightarrow \ \frac{1}{R^2 - (x_*^2+y_*^2)} + \frac{x_*^2+y_*^2}{\bigg(R^2 - (x_*^2+y_*^2)\bigg)^2} > \left(u-\frac{-x_*}{R^2-(x_*^2+y_*^2)}\right)^2 + \left(v - \frac{y_*}{R^2-(x_*^2+y_*^2)}\right)^2 \\
&\Longleftrightarrow \ \frac{1}{R^2 - (x_*^2+y_*^2)} > u^2 -  2u\frac{-x_*}{R^2 - (x_*^2+y_*^2)}+v^2 - 2v\frac{y_*}{R^2 - (x_*^2+y_*^2)} \\
&\Longleftrightarrow \ \frac{1}{R^2 - (x_*^2+y_*^2)} + 2u\frac{-x_*}{R^2 - (x_*^2+y_*^2)} + 2v\frac{y_*}{R^2 - (x_*^2+y_*^2)} > u^2+v^2 \\
&\Longleftrightarrow \ 1 - 2ux_* + 2vy_* < \bigg(R^2 - (x_*^2+y_*^2)\bigg)(u^2+v^2) \\
&\Longleftrightarrow \ 1 - 2ux_* + 2vy_* +(x_*^2+y_*^2)(u^2+v^2) < R^2(u^2+v^2).  
\end{split}
\label{estimateupsidedown}
\end{equation}

\noindent
Consequently, we conlcude that
\begin{equation}
\begin{split}
&(x-x_*)^2 + (y-y_*)^2 = \left(\frac{u}{u^2+v^2}-x_*\right)^2 + \left(\frac{v}{u^2+v^2}-y_*\right)^2\\
&= \left(\frac{u-x_*(u^2+v^2)}{u^2+v^2}\right)^2 + \left(\frac{v-y_*(u^2+v^2)}{u^2+v^2}\right)^2 \\
&= \frac{(u^2+v^2) + 2(vy_*-ux_*)(u^2+v^2) + (x_*^2+y_*^2)(u^2+v^2)^2}{(u^2+v^2)^2} \\
&= \frac{1 + 2(vy_*-ux_*) + (x_*^2+y_*^2)(u^2+v^2)}{u^2+v^2} \stackrel{(\ref{estimateupsidedown})}{<} R^2,
\end{split}
\label{xminusxstar2yminusystar2}
\end{equation}

\noindent
which ends the first part of the theorem. As we remarked earlier, the reasoning behind the second part is completely analogous.
\end{proof}

In view of the above theorem, it is a trivial observation that
$$\iota(\overline{B}(z_*,r)) = \overline{B}\left(\frac{-x_*}{R^2-(x_*^2+y_*^2)}+i\frac{y_*}{R^2-(x_*^2+y_*^2)},\frac{R}{R^2-(x_*^2+y_*^2)}\right) \hspace{0.4cm}\text{if}\hspace{0.4cm} 0\not\in \overline{B}(z_*,R),$$

\noindent
and  
$$\iota(\overline{B}(z_*,r)) = \overline{\complex} \backslash B\left(\frac{-x_*}{R^2-(x_*^2+y_*^2)}+i\frac{y_*}{R^2-(x_*^2+y_*^2)},\frac{R}{R^2-(x_*^2+y_*^2)}\right) \hspace{0.4cm}\text{if}\hspace{0.4cm} 0 \in B(z_*,R).$$

Up to this point, we did not allow $0$ to lay on the boundary of the circle. The next result investigates exactly that scenario. 

\begin{thm}
If $0\in S(z_*,R)$, then $\iota(B(z_*,R))$ is an open half-plane, described by the inequality 
\begin{gather}
x_*x - y_*y > \frac{1}{2}.
\label{xstarxystary}
\end{gather}
\end{thm}
\begin{proof}
Let $x+iy\in B(z_*,R)$. Since 
$$x_*^2+y_*^2 = R^2 > (x-x_*)^2 + (y-y_*)^2,$$

\noindent
then 
\begin{gather*}
2x_*x + 2y_*y > x^2 + y^2 \ \Longleftrightarrow \ x_*\frac{x}{x^2 + y^2} - y_*\frac{-y}{x^2 + y^2} > \frac{1}{2}. 
\end{gather*}

\noindent
This proves that $\iota(B(z_*,R))$ is contained in the open half-plane (\ref{xstarxystary}). 

We will show that the whole open half-plane is the image of $B(z_*,R)$ under $\iota$. Let $u+iv$ belong to an open half-plane $(\ref{xstarxystary})$, meaning
\begin{gather}
x_*u-y_*v > \frac{1}{2} \ \Longleftrightarrow \ 1 - 2x_*u + 2y_*v <0.
\label{12xstaru2ystarv0} 
\end{gather}

\noindent
Put
$$x = \frac{u}{u^2+v^2} \hspace{0.4cm}\text{and}\hspace{0.4cm} y = \frac{-v}{u^2+v^2}.$$

\noindent
As in (\ref{xminusxstar2yminusystar2}), we gather that 
\begin{gather*}
(x-x_*)^2 + (y-y_*)^2 = \frac{1 + 2(vy_*-ux_*) + (x_*^2+y_*^2)(u^2+v^2)}{u^2+v^2} \\
\frac{1 + 2(vy_*-ux_*)}{u^2+v^2} + R^2\stackrel{(\ref{12xstaru2ystarv0})}{<} R^2,
\end{gather*}

\noindent
which concludes the proof. 
\end{proof}

It is now an easy remark that $\iota(\overline{B}(z_*,R))$ is the closed half-plane, described by the inequality
$$x_*x-y_*y \geq \frac{1}{2}.$$

\section{Inverting lines and half-planes}

We begin this chapter with inverting lines. For convenience, let us denote \mbox{$\overline{\complex^*} = \overline{\complex}\backslash\{0\}$}. As the result below shows, there are various scenarios to consider. 

\begin{thm}
Let $L$ be a line described by:
\begin{itemize}
	\item $y = ax + b,$ where $b\neq 0$, then $\iota(L)$ is a circle described by the equation
	$$\left(x+\frac{a}{2b}\right)^2 + \left(y+\frac{1}{2b}\right)^2 = \left(\frac{\sqrt{1+a^2}}{2b}\right)^2,$$
	
	\item $y = ax,$ then 
	$$h(L) = \bigg\{z\in \overline{\complex^*}\ :\ y=-ax\bigg\},$$
	\item $x = c\neq 0,$ then $h(L)$ is a circle described by the equation
	$$\left(x-\frac{1}{2c}\right)^2 + y^2 = \left(\frac{1}{2c}\right)^2,$$
	\item $x = 0,$ then 
	$$h(L) = \bigg\{z\in\overline{\complex^*}\ :\ x=0\bigg\}.$$
\end{itemize}
\end{thm}
\begin{proof}
At first, suppose that $L$ is described by the equation $y=ax+b$. This implies that if $z\in L$ then 
\begin{gather*}
\iota(z) = \frac{1}{x+i(ax+b)} = \frac{x}{x^2+(ax+b)^2} + i\frac{-(ax+b)}{x^2+(ax+b)^2}.
\end{gather*}

\noindent
Let us observe that 
\begin{gather*}
\left(\frac{x}{x^2+(ax+b)^2} + \frac{a}{2b}\right)^2 + \left(\frac{ax+b}{x^2+(ax+b)^2} - \frac{1}{2b}\right)^2 \\
= \frac{\bigg(2bx+a(x^2+(ax+b)^2)\bigg)^2 + \bigg(2b(ax+b) - x^2 - (ax+b)^2\bigg)^2}{4b^2\bigg(x^2+(ax+b)^2\bigg)^2} \\
= \frac{\bigg(a(1+a^2)x^2 + 2(1+a^2)bx + ab^2\bigg)^2 + \bigg(-(1+a^2)x^2 +b^2\bigg)^2}{4b^2\bigg(x^2+(ax+b)^2\bigg)^2} \\
= \frac{(1+a^2)\bigg((1+a^2)^2x^4 + 4a(1+a^2)bx^3 + 2(1+3a^2)b^2x^2 + 4ab^3x + b^4\bigg)}{4b^2\bigg(x^2+(ax+b)^2\bigg)^2} = \frac{1+a^2}{4b^2},
\end{gather*}

\noindent
which proves the first part of the theorem. 

In order to prove the second part of the theorem, let $L$ be described by the equation $y = ax$. Consequently, we have 
$$\iota(z) = \frac{1}{x+iax} = \frac{x-iax}{(1+a^2)x^2} = \frac{1}{(1+a^2)x} + i\frac{-a}{(1+a^2)x}$$

\noindent
and it remains to remark that 
$$a\cdot \frac{1}{(1+a^2)x} + 1\cdot \frac{-a}{(1+a^2)x} = 0.$$

For the third part of the theorem, let $L$ be described by the equation $x = c \neq 0$. Hence, we gather that 
$$\iota(z) = \frac{1}{c+iy} = \frac{c-iy}{c^2+y^2} = \frac{c}{c^2+y^2} + i\frac{-y}{c^2+y^2},$$

\noindent
which implies that 
\begin{gather*}
\left(\frac{c}{c^2+y^2} - \frac{1}{2c}\right)^2 + \left(\frac{-y}{c^2+y^2}\right)^2 = \left(\frac{c^2-y^2}{2c(c^2+y^2)}\right)^2 + \frac{y^2}{c^2+y^2} \\
= \frac{(c^2-y^2)^2 + 4c^2y^2}{4c^2(c^2+y^2)^2} = \frac{1}{4c^2} = \left(\frac{1}{2c}\right)^2.
\end{gather*}

\noindent
This ends the third part of the theorem. 

Last but not least, assume that $L$ is described by the equation $x=0$. We may thus parametrize it by $z = iy,$ where $y\in\reals$. Consequently, we have 
$$\iota(z) = \frac{1}{iy} = i\frac{-1}{y},$$

\noindent
which ends the proof. 
\end{proof}

Before we proceed with establishing the images of open half-planes, we require two technical lemmas, which we present below.

\begin{lemma}
Let $z_*$ belong to an open half-plane $P$ described either by the inequality \mbox{$y>ax+b$}, $b>0$ or the inequality $y<ax+b,\ b<0$. Then 
$$\iota(z_*) \in B\left(\frac{-a}{2b} + i\frac{-1}{2b}, \frac{\sqrt{a^2+1}}{2b}\right).$$
\label{halfplanePeitheror}
\end{lemma}
\begin{proof}
Since $z_* = x_*+iy_*$ belongs to a half-plane $P$, we know that $y=ax + c,$ where either $c > b > 0$ or $c < b < 0$. Since 
$$\frac{1}{z_*} = \frac{x_*}{x_*^2+(ax_*+c)^2} + i\frac{-(ax_*+c)}{x_*^2+(ax_*+c)^2}$$

\noindent
then
\begin{gather*}
\left(\frac{x_*}{x_*^2+(ax_*+c)^2} + \frac{a}{2b}\right)^2 + \left(\frac{-(ax_*+b)}{x_*^2+(ax_*+b)^2} + \frac{1}{2b}\right)^2 \\
= \left(\frac{2bx_* + a\bigg(x_*^2+(ax_*+c)^2\bigg)}{2b\bigg(x_*^2+(ax_*+c)^2\bigg)}\right)^2 + \left(\frac{-2b(ax_*+c) + \bigg(x_*^2+(ax_*+c)^2\bigg)}{2b\bigg(x_*^2+(ax_*+c)^2\bigg)}\right)^2 \\
= \frac{4b^2\bigg(x_*^2+(ax_*+c)^2\bigg) - 4bc\bigg(x_*^2+(ax_*+c)^2\bigg) + (1+a^2)\bigg(x_*^2+(ax_*+c)^2\bigg)^2}{4b^2\bigg(x_*^2+(ax_*+c)^2\bigg)^2} \\
= \frac{4b^2-4bc+(a^2+1)\bigg(x_*^2+(ax_*+c)^2\bigg)}{4b^2\bigg(x_*^2+(ax_*+c)^2\bigg)} < \frac{1+a^2}{4b^2}.
\end{gather*}

\noindent
The last inequality stems from $4b^2 < 4bc$, which is true in both scenarios: $c>b>0$ and $c<b<0$. This ends the proof.
\end{proof}

\begin{lemma}
Let $z_* = x_* + iy_*$ belong to an open half-plane $P$ described either by the inequality \mbox{$x>c>0$}, or the inequality $x<c<0$. Then 
$$\iota(z_*) \in B\left(\frac{1}{2c}, \frac{1}{2c}\right).$$
\label{iotazstarB12c12c}
\end{lemma}
\begin{proof}
Since 
$$\iota(z_*) = \frac{x_*}{x_*^2+y_*^2} + i \frac{-y_*}{x_*^2+y_*^2},$$

\noindent
then
\begin{gather*}
\left(\frac{x_*}{x_*^2+y_*^2}-\frac{1}{2c}\right)^2 + \left(\frac{-y_*}{x_*^2+y_*^2}\right)^2 = \left(\frac{2cx_*-(x_*^2+y_*^2)}{2c(x_*^2+y_*^2)}\right)^2 + \frac{y_*^2}{(x_*^2+y_*^2)^2} \\
= \frac{4c^2(x_*^2+y_*^2) - 4cx_*(x_*^2+y_*^2) + (x_*^2+y_*^2)^2}{4c^2(x_*^2+y_*^2)^2} = \frac{4c^2-4cx_* + (x_*^2+y_*^2)}{4c^2(x_*^2+y_*^2)} < \frac{1}{4c^2}.
\end{gather*}

\noindent
The last inequality stems from $4c^2 < 4cx_*$, which is true in both scenarios: $x_*>c>0$ and $c<x_*<0$. This ends the proof.
\end{proof}

With the above two technical lemmas, we are in position to describe the inverse image of an open half-plane, which does not contain $0$ on its boundardy.

\begin{thm}
Let $P$ be an open half-plane such that $0\not\in \overline{P}$.
\begin{itemize}
	\item If $P$ is described either by $y>ax+b,\ b>0$ or $y<ax+b,\ b<0$ then 
	\begin{gather}
	\iota(P) = B\left(\frac{-a}{2b} + i\frac{-1}{2b}, \frac{\sqrt{a^2+1}}{2b}\right).
	\label{Ba2bi12ba212b}
	\end{gather}
	
	\item If $P$ is described either by $x>c>0$ or $x<c<0$ then 
	$$\iota(P) = B\left(\frac{1}{2c},\frac{1}{2c}\right).$$
\end{itemize}
\label{halfplanewithoutzero}
\end{thm}
\begin{proof}
At first, let us focus on the case when the half-plane $P$ is described either by $y>ax+b,\ b>0$ or $y<ax+b,\ b<0$. By Lemma \ref{halfplanePeitheror}, the image of $P$ lies inside the ball (\ref{Ba2bi12ba212b}). We prove that the half-plane is in fact mapped onto this ball. 

Suppose that $z_* = u_*+iv_*$ is an element of the ball (\ref{Ba2bi12ba212b}) and put
\begin{gather}
x_* = \frac{u_*}{u_*^2+v_*^2} \hspace{0.4cm}\text{and}\hspace{0.4cm} y_* = \frac{-v_*}{u_*^2+v_*^2}.
\label{definitionofxstarystar}
\end{gather}

\noindent
Note that the point $(x_*,y_*)$ lies on the line
$$y = ax - \frac{v_*+au_*}{u_*^2+v_*^2}.$$

\noindent
Furthermore, the fact that $(u_*,v_*)$ lies in the ball (\ref{Ba2bi12ba212b}) is equivalent to 
$$y_*-ax_* = -\frac{v_*+au_*}{u_*^2+v_*^2} > b \hspace{0.4cm}\text{if}\hspace{0.4cm} b>0,$$

\noindent
or
$$y_*-ax_* = -\frac{v_*+au_*}{u_*^2+v_*^2} < b \hspace{0.4cm}\text{if}\hspace{0.4cm} b<0.$$

\noindent
Either way, this means that $(x_*,y_*)$ lies in the half-plane $P$.

We proceed with proving the second part of the theorem. Suppose that $P$ is described either by $x>c>0$ or $x<c<0$. By Lemma \ref{iotazstarB12c12c}, the image of $P$ lies inside the ball $B\left(\frac{1}{2c},\frac{1}{2c}\right)$. We prove that the half-plane is in fact mapped onto this ball.

Suppose that $z_* = u_* + iv_*$ lies inside $B\left(\frac{1}{2c},\frac{1}{2c}\right)$. We define $x_*,y_*$ again by (\ref{definitionofxstarystar}). Finally, the fact that $(u_*,v_*)$ lies in $B\left(\frac{1}{2c},\frac{1}{2c}\right)$ is equivalent to 
$$x_* = \frac{u_*}{u_*^2+v_*^2} > c \hspace{0.4cm}\text{if}\hspace{0.4cm} c>0,$$

\noindent
or
$$x_* = \frac{u_*}{u_*^2+v_*^2} < c \hspace{0.4cm}\text{if}\hspace{0.4cm} c<0.$$
 
\noindent
We conclude that $(x_*,y_*)$ belongs to $P$, which ends the proof. 
\end{proof}

Naturally (still working under the assumption that $0\not\in\overline{P}$), if $P$ is described either by $y\geq ax+b,\ b>0$ or $y\leq ax+b,\ b<0$ then
\begin{gather}
\iota(P) = \overline{B}\left(\frac{-a}{2b} + i\frac{-1}{2b}, \frac{\sqrt{a^2+1}}{2b}\right).
\end{gather}
	
\noindent
Furthermore, if $P$ is described either by $x\geq c>0$ or $x\leq c<0$ then 
$$\iota(P) = \overline{B}\left(\frac{1}{2c},\frac{1}{2c}\right).$$
		
The theorem below can be treated as a kind of 'mirror image' to Theorem \ref{halfplanewithoutzero}. This duality is reflected in an immediate proof. 

\begin{thm}
Let $P$ be an open half-plane such that $0\in P$.
\begin{itemize}
	\item If $P$ is described either by $y<ax+b,\ b>0$ or $y>ax+b,\ b<0$ then
	\begin{gather*}
	\iota(P) = \overline{\complex^*}\backslash \overline{B}\left(\frac{-a}{2b} + i\frac{-1}{2b}, \frac{\sqrt{a^2+1}}{2b}\right).
	\end{gather*}
	
	\item If $P$ is described either by $x<c$, where $c>0$ or $x>c$, where $c<0$ then 
	$$\iota(P) = \overline{\complex^*}\backslash \overline{B}\left(\frac{1}{2c},\frac{1}{2c}\right).$$
\end{itemize}
\end{thm}
\begin{proof}
It suffices to observe that $Q = \complex\backslash\overline{P}$ is an open half-plane satisfying the corresponding assumptions in Theorem \ref{halfplanewithoutzero}.
\end{proof}

In view of the above theorem (still assuming that $0\in P$), if $P$ is described either by $y\leq ax+b,\ b>0$ or $y\geq ax+b,\ b<0$ then  
\begin{gather*}
\iota(P) = \overline{\complex^*}\backslash B\left(\frac{-a}{2b} + i\frac{-1}{2b}, \frac{\sqrt{a^2+1}}{2b}\right).
\end{gather*}
	
\noindent
Furthermore, if $P$ is described either by $x\leq c$, where $c>0$ or $x\geq c$, where $c<0$ then
$$\iota(P) = \overline{\complex^*}\backslash B\left(\frac{1}{2c},\frac{1}{2c}\right).$$
		
The remaining two results concern the scenario, when $0$ lies on the boundary of the half-plane $P$.

\begin{lemma}
Let $z_* = x_*+iy_*$ belong to an open half-plane $P$ such that $0\in\partial P$.
\begin{itemize}
	\item If $P$ is described either by $y>ax$ or $y<ax$ then $\iota(z_*)$ belongs to the open half-plane described by $y<-ax$ or $y>-ax$, respectively. 
	\item If $P$ is described either by $x>0$ or $x<0$, then $\iota(z_*) \in P$. 
\end{itemize}
\label{halfplane0ontheboundary}
\end{lemma}
\begin{proof}
Since 
$$\iota(z_*) = \frac{x_*}{x_*^2+y_*^2} + i \frac{-y_*}{x_*^2+y_*^2}$$

\noindent
then 
$$\frac{-y_*}{x_*^2+y_*^2} + a\frac{x_*}{x_*^2+y_*^2} < 0 \hspace{0.4cm}\text{if}\hspace{0.4cm} y_*>ax_*.$$

\noindent
This means that $\iota(z_*)$ lies in the half-plane $y<-ax$. The case when $P$ is described by $y<ax$ is treated analogously. 

The second part of the theorem is trivial.
\end{proof}

\begin{thm}
Let $P$ be an open half-plane such that $0\in \partial P$.
\begin{itemize}
	\item If $P$ is described either by $y>ax$ or $y<ax$ then $\iota(P)$ is the open half-plane described by $y>-ax$ or $y<-ax$, respectively. 
	\item If $P$ is described either by $x>0$ or $x<0$, then $\iota(P) = P$. 
\end{itemize}
\end{thm}
\begin{proof}
At first, suppose that $P$ is described by $y>ax$. By Lemma \ref{halfplane0ontheboundary}, we know that $\iota(P)$ lies inside the open half-plane described by $y<-ax$. We show that $\iota(P)$ is the whole half-plane. Suppose that the point $u_*+iv_*$ satisfies $v_*<-au_*$. If we define $x_*,y_*$ by (\ref{definitionofxstarystar}), then 
$$y_*-ax_* = -\frac{v_*+au_*}{u_*^2+v_*^2} > 0.$$

\noindent
Thus $x_*+iy_*$ lies in $P$, which we aimed to prove. The case when $P$ is described by $y<ax$ is treated analogously. 

The second part of the theorem is trivial.
\end{proof}

Last but not least (assuming $0\in \partial P$) if $P$ is described either by $y\geq ax$ or $y\leq ax$ then 
$$\iota(P) = \bigg\{z\in\overline{\complex^*}\ :\ y\geq -ax\bigg\}$$ 

\noindent
or 
$$\iota(P) = \bigg\{z\in\overline{\complex^*}\ :\ y\leq -ax\bigg\},$$ 

\noindent
respectively. Furthermore, if $P$ is described either by $x\geq 0$ or $x\leq 0$, then 
$$\iota(P) = P\cap\overline{\complex^*}.$$

\end{document}